\documentclass[10pt,twoside,leqno]{siamltex}
\usepackage{amsfonts,epsfig,comment}

\def\cvd{~\vbox{\hrule\hbox{%
     \vrule height1.3ex\hskip0.8ex\vrule}\hrule } }

\newtheorem{remark}[theorem]{Remark}

\newcommand{\reals}{\mathbb{R}}
\newcommand{\complex}{\mathbb{C}}

\author{
Roderick Edwards\thanks{Department of Mathematics and Statistics, University of Victoria, Victoria BC V8W 3R4, Canada (edwards@uvic.ca).}
\and
Eric Foxall\thanks{Department of Mathematics and Statistics, University of Victoria, Victoria BC V8W 3R4, Canada (e.t.foxall@gmail.com).}
\and 
Theodore J. Perkins\thanks{Ottawa Hospital Research Institute, Ottawa ON K1H 8L6, Canada
(tperkins@ohri.ca), and Department of Biochemistry, Microbiology and
Immunology, University of Ottawa, Ottawa ON K1H 8M5, Canada.}}

\usepackage[margin=1in]{geometry}
\usepackage{amsmath}
\usepackage{graphicx}
\usepackage{subfigure}
\usepackage{amssymb}
\usepackage{mathrsfs}
\DeclareMathOperator{\itin}{It}

\DeclareMathOperator{\vl}{Vl}
\DeclareMathOperator{\vt}{Vt}

\DeclareMathOperator{\spn}{span}

\title{Scaling Properties of Paths on Graphs}
\begin{document}
\maketitle

\section*{Abstract}
Let $G$ be a directed graph on finitely many vertices and edges, and assign a positive weight to each edge on $G$.  Fix vertices $u$ and $v$ and consider the set of paths that start at $u$ and end at $v$, self-intersecting in any number of places along the way.  For each path, sum the weights of its edges, and then list the path weights in increasing order.  The asymptotic behaviour of this sequence is described, in terms of the structure and type of strongly connected components on the graph.  As a special case, for a Markov chain the asymptotic probability of paths obeys either a power law scaling or a weaker type of scaling, depending on the structure of the transition matrix.  This generalizes previous work by Mandelbrot and others, who established asymptotic power law scaling for special classes of Markov chains.

\begin{keywords}
Non-negative matrices, Perron-Frobenius theory, Directed graphs, Markov chains, Power law scaling.
\end{keywords}
\begin{AMS}
15B48, 60J10. 
\end{AMS}

\section{Introduction}
Many sequential processes can be described as walks on directed graphs.  Consider examples such as one's morning drive to work, or navigating the world-wide web, or stochastic conformational changes in a protein molecule, or fluctuations in the value of a stock on the stock market.  
In each case there is a natural notion of ``state" to the system, which can be viewed abstractly as a vertex in a graph: one can be at a particular intersection in the city, one can be viewing a particular page on the world-wide web, the protein molecule can be in a particular conformation, and the stock has a current price.\\

Moreover, there are transitions between states that can be can viewed as edges on the graph: roads allow us to travel between intersections, hyperlinks allow navigation between web pages, thermal fluctuations cause a molecule to switch from one conformation to another, and buying or selling pressure can change the price of a stock.  In general, these links may be unidirectional.  For instance, some roads allow travel in only one direction.  On the world wide web, one web page may link to a second page, but the second page may have no link back to the first.\\

Now, suppose we attach a positive weight to each edge in the graph.  In a road network example, where each edge corresponds to a stretch of road, we might associate to each edge the length of the corresponding road, or the amount of time it takes to travel that  road.  Then, the total distance travelled or time taken in travelling any particular route from home to work is equal to the sum of weights of the corresponding edges.  In the stochastic molecule scenario, associate to each edge the negative log probability of the corresponding change occurring, which is a positive number if the probability of change is less than $1$.  Then, the negative log probability of any sequence of conformational changes is again given by the sum of weights of the corresponding edges.\\

In general, there may be many paths between two vertices in a directed graph.  Indeed, if one allows paths to visit the same vertex more than once, then there are in general infinitely many possible paths, even if the graph itself is finite.\\


Among all the possible paths between two vertices on a weighted directed graph, one will have minimum total weight - corresponding to the shortest or fastest route to work, or the most probable sequence of steps from one molecular state to another.  Another path will have the second smallest total weight, another will have the third smallest, and so on.  This begs the question: How does this sequence of weights behave asymptotically? 
More formally, if we let $p_r$ be the weight of the path with $r^{th}$ smallest total weight, how does $p_r$ scale with $r$?  This is the question answered in this paper.\\

We show that the order of this relationship depends only on the structure and type of strongly connected components in the graph, while the exact rate of scaling depends on the edge weights as well.  We also show how to compute the scaling relationship for any given instance using standard graph-theoretic algorithms and eigenvalue computations.

\section{Main Result}\label{secmain}

The main result of the paper is Theorem \ref{main}.  First, we establish some language for describing paths and path weights on a directed graph.

\subsection{Paths}\label{secpaths}
Let $G = (V,E,I,O,W)$ denote an edge-weighted directed multigraph (i.e., a graph in which multiple edges may emanate from a vertex), where $V$ and $E$ are finite sets and $I:E\rightarrow V$, $O:E\rightarrow V$ and $W:E\rightarrow \reals^+$ are functions.  The set $V$ is called the vertex set, and $E$ is the edge set; if $I(e) = u$ and $O(e)=v$ then $e$ is an edge from $u$ to $v$; $W(e)$ denotes the weight of the edge.  For $u,v \in V$, $E(u,v)$ denotes the set of edges from $u$ to $v$.  Note that each subset $U\subset V$ \emph{induces} a graph defined by restricting to the vertex set $U$ and to the edges that satisfy $I(e)\in U$, $O(e)\in U$.  The in-degree of a vertex $v$ is the cardinality of $\{e \in E: O(e)=v\}$, and the out-degree of $v$ is the cardinality of $\{e \in E: I(e)=v\}$.\\

A path on $G$ is a non-empty list of edges $x = x_1x_2...x_k$, $x_i \in E$ for $1 \leq i \leq k$, such that $I(x_{i+1}) = O(x_i)$, $1 \leq i < k$.  Say that $x$ is a path from $u$ to $v$ and write $I(x)=u$, $O(x)=v$ if $I(x_1)=u$ and $O(x_k)=v$.  For $u,v \in V$ say that $u\rightarrow v$ if there is a path from $u$ to $v$, and say that $u\leftrightarrow v$ if $u \rightarrow v$ and $v\rightarrow u$.  Let $[u] = \{v \in V: u\ \leftrightarrow v\}$.  Since $\leftrightarrow$ is symmetric and transitive, it partitions $\{v \in V: [v]\neq\emptyset\}$ into classes, which are called the strongly connected components of the graph.  A graph is said to be strongly connected if $u\leftrightarrow v$ for each pair $u,v$ of vertices on the graph.  See Figure \ref{figsccs} for an example.  A \emph{cycle} is a strongly connected graph in which every vertex has in-degree and out-degree equal to 1.\\

\begin{figure}
\centering
\mbox{\subfigure[A graph having no strongly connected components]{\includegraphics[height=60mm,width=75mm]{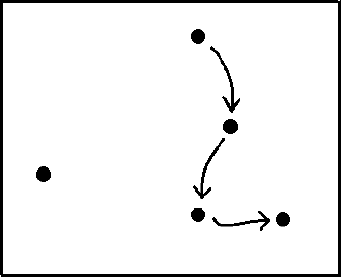}}\quad
\subfigure[A graph having two strongly connected components:  the singleton on the left, and the two vertices on the lower right]{\includegraphics[height=60mm,width=75mm]{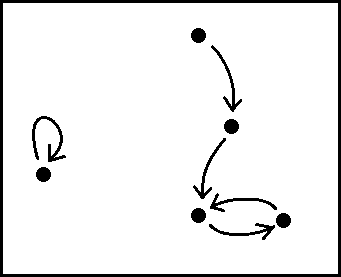} }}
\caption{Two examples of graphs and their strongly connected components.}
\label{figsccs}
\end{figure}

Let $\vt(x) = \{v \in V: v=I(x_i)\,\,\textrm{or}\,\,v=O(x_i)\,\,\textrm{for some}\,\,i\}$ denote the set of vertices met by a path $x$, and let $l(x)$, the \emph{length} of a path, denote the number of edges on that path; for example, if $x=x_1...x_k$ then $l(x)=k$.  For a set of vertices $U \subset V$, say that $x$ is a path on $U$ if $\vt(x) \subset U$.  Let $W(x) = \sum_i W(x_i)$ denote the weight of a path.  If $x = x_1...x_j$ is a path from $v_1$ to $v_2$ and $y = y_1...y_k$ is a path from $v_2$ to $v_3$ then $xy = x_1...x_jy_1...y_k$ is a path from $v_1$ to $v_3$ and $W(xy) = W(x)+W(y)$.  For any set of paths $X$, let $\vt(X) = \{\vt(x):x \in X\}$, then every path in $X$ is a path on $\vt(X)$.\\

Let $W(X) = \bigsqcup_{x \in X}W(x)$; $W(X)$ is called the \emph{set of weights} for $X$.  The \emph{sequence of weights} (s.o.w.) $(p_r)$ for a set of paths $X$, or more accurately for the set of weights $W(X)$, is an enumeration of the elements of $W(X)$ in ascending order.  The subscript $_r$ in $(p_r)$ is called the \emph{rank} of a path.\\

Although denoted $(p_r)$, the sequence of weights is a sequence of positive numbers and not probabilities (the lower-case $w$ is reserved for vectors).  However, a Markov chain can easily be converted to a graph of the above type by collapsing pairs of nodes linked by edges of probability 1, and then taking negative $\log$ of the probabilities.  Moreover, the weight of a path is then equal to negative $\log$ of its probability, since $\log$ takes products to sums.\\

The following is the main result of this paper.  Sections \ref{secitin} and \ref{secdecomp} should suffice to explain how the result is obtained from the Lemmata and Theorems mentioned in the statement of the result.\\

\begin{theorem}\label{main} Let $G=(V,E,I,O,W)$ denote a directed weighted graph.  For $v_1,v_2 \in V$, let $X$ denote the set of paths from $v_1$ to $v_2$ on $G$, and let $(p_r)$ denote the corresponding sequence of weights.  Suppose $X$ is non-empty.
\begin{remunerate}
\item If there are no strongly connected components (s.c.c's) on $\vt(X)$ then $X$ is a finite set.
\item If every s.c.c. on $\vt(X)$ is a cycle, then let $c$ be the greatest number of components met by a path, and
\begin{equation*}
\lim_{r\to\infty}p_r^c/r = s
\end{equation*}
where the value of $s$ is computed from the structure of s.c.c.'s on $\vt(X)$, using Lemma \ref{cycle}, Lemma \ref{comp} and Lemma \ref{union}.
\item If there is at least one s.c.c. on $\vt(X)$ which is not a cycle then
\begin{equation*}
\lim_{r\to\infty}p_r/\log r = s
\end{equation*}
where $s$ is the smallest value assigned to a s.c.c. on $\vt(X)$ by Theorem \ref{irredGraph}.
\end{remunerate}
\end{theorem}

\begin{remark}The asymptotic behaviour of more general classes of paths can be computed using the above result.  For instance, the set of paths from a fixed vertex to an arbitrary vertex is the disjoint union of such sets, and the set of paths from a fixed vertex, passing through a second fixed vertex, to a third fixed vertex, is a direct sum of sets of this type.  Moreover, the rules for computing the asymptotics of these sets are given by Lemma \ref{comp} and Lemma \ref{union}.
\end{remark}

\subsection{The Itinerary}\label{secitin}
The classification in \ref{main} is enabled by a function called the itinerary, defined below.  The itinerary of a path is a partial description of the path; it gives the start vertex of the path, the end vertex of the path, and for each s.c.c. met by the path, it gives the entry and exit vertices to the s.c.c.  For the following, define $\vl(x) = I(x_1)I(x_2)...I(x_k)O(x_k)$ that lists the vertices met by a path.

\begin{definition} \rm
For a path $x$ let $s_1...s_k$ denote $\vl(x)$.  For $1<i<k$, if $s_{i-1}\leftrightarrow s_i \leftrightarrow s_{i+1}$, substitute $s_{i-1}s_{i+1}$ for $s_{i-1}s_is_{i+1}$.  Since the substitution shortens the list, the process terminates in a list
\begin{equation*}
\itin(x) = s_1...s_m
\end{equation*}
which is called the \emph{itinerary} of $x$.
\end{definition}

If $s_1...s_m$ is an itinerary then $s_i \leftrightarrow s_j \Rightarrow |i-j| \leq 1$; in other words, only the entry and exit vertices to each strongly connected component met by a path appear in the itinerary.  This is because, since $s_i \rightarrow s_j$ for $i<j$ and $u \rightarrow u \Rightarrow u\leftrightarrow u$ for each $u \in V$, so that from the construction, $s_i \neq s_j$ for $j>i+1$.  A corollary of this construction is that a vertex appears at most twice in a given itinerary, and so the cardinality of the range of $\itin$ is bounded by $\leq (2|V|)!$, and in particular is finite.\\

The next lemma states that if $s$ is the itinerary of a path in $X(v_1,v_2)$, then $It^{-1}(s)$, the set of paths in $X(v_1,v_2)$ with itinerary $s$, is a direct product of paths on s.c.c.'s and transitions from one s.c.c. to the next.  For $U\subset V$ and $v_1,v_2 \in U$ let $X(v_1,v_2;U)$ denote $\{x \in X(v_1,v_2): \vt(x) \subset U\}$, the set of paths from $v_1$ to $v_2$ on $U$.  A list $s_1...s_m$ is an \emph{admissible} itinerary for a set $X$ if there exists $x \in X$ such that $\itin(x) = s_1...s_m$.\\

\begin{lemma}\label{itin}For $v_1,v_2 \in V$, let $s=s_1...s_m$ be an admissible itinerary for $X(v_1,v_2)$.  Then $\itin^{-1}(s)$ is the set of paths of the form $x^{(1)}...x^{(m-1)}$, where $x^{(i)} \in X(s_i,s_{i+1};[s_i])$ if $s_i\leftrightarrow s_{i+1}$ and $x^{(i)} \in E(s_i,s_{i+1})$ otherwise.
\end{lemma}
\begin{proof}Let $x \in X(v_1,v_2)$ and let $u_1...u_k$ denote $\vl(x)$ and $s_1...s_m$ denote $\itin(x)$.  From the definition of $\itin$ there is a strictly increasing function $\sigma:\{1,...,m\} \rightarrow \{1,...,k\}$ such that $u_{\sigma(i)} = s_i$ for $1\leq i \leq m$, and such that $u_j \in [u_{\sigma(i)}]$ for $\sigma(i) \leq j \leq \sigma(i+1)$ if $s_i \leftrightarrow s_{i+1}$, and $\sigma(i)+1 = \sigma(i+1)$ otherwise.  Therefore $x$ has the form described above.  Conversely, each path of the form described above has itinerary $s_1...s_m$.  If $s_1...s_m$ is admissible for $X(v_1,v_2)$ then $s_1=v_1$ and $s_m=v_2$, so that each path of the form described above is a path from $v_1$ to $v_2$.
\end{proof}

\begin{remark}
If $v_1=v_2=v$ then $\itin(x)=vv$ for each $x \in X(v_1,v_2)$; to see this let $s_1...s_k=\vl(x)$.  Then $s_1 = s_k=v$ and $v=s_1 \rightarrow s_i \rightarrow s_k = v$ for $1<i<k$, so that $v \leftrightarrow s_i$ for $1 \leq i \leq k$, and by transitivity, $s_i\leftrightarrow s_j$ for $1\leq i,j \leq k$ and all but the endpoints are collapsed.
\end{remark}

At this point we can prove Part 1 of \ref{main}.
\begin{corollary}For a graph $G=(V,E,I,O,U)$ and $v_1,v_2 \in V$, let $X$ denote the set of paths from $v_1$ to $v_2$.  If there are no strongly connected components on $\vt(X)$, then $X$ is a finite set.
\end{corollary}
\begin{proof}Observe that $X = \bigcup_{s \in \itin(X)}\itin^{-1}(s)$.  Since there are no s.c.c., for each $u,v \in \vt(X)$, $u\nleftrightarrow v$.  Let $s=s_1...s_m$ be an admissible itinerary, then $s_i \in \vt(X)$, $1\leq i \leq m$ and so $s_i \nleftrightarrow s_j$, $1\leq i,j \leq m$.  Therefore, $\itin^{-1}(s)$ is the set of paths of the form $x^{(1)}...x^{(m-1)}$, where $x^{(i)} \in E(s_i,s_{i+1})$ for $1 \leq i <m$.  Since for each $u,v \in V$, $E(u,v)$ is a finite set, $\itin^{-1}(s)$ is finite for each $s \in \itin(X)$.  Since $\itin(X)$ is a finite set it follows that $X$ is a finite set.
\end{proof}

If there are strongly connected components on $\vt(X(v_1,v_2))$, then $X(v_1,v_2)$ is an infinite set, since it is possible to cycle around on an s.c.c. and obtain longer and longer paths.

\subsection{Decomposition of the sequence of weights}\label{secdecomp}

The following definitions are used to describe the forthcoming decomposition.
\begin{definition} \rm
Suppose for each $i\in\{1,\ldots,k\}$ that $(p_{r}^{(i)})$, $r=1,2,...$ is a non-decreasing positive sequence.  The \emph{composition} of the sequences $(p_{r}^{(i)})$, $i=1,...,k$, is the unique (up to permutation of equal entries) non-decreasing sequence containing the entries $\sum_{i=1}^k p_{j_i}^{(i)}$, where $(j_i)$ ranges over $\mathbb{N}^k$.
\end{definition}

If for $1 \leq i \leq k$, $X_i$ is a set of paths and $(p_r^{(i)})$ is the sequence of weights for $W(X_i)$, then the sequence of weights for $\bigoplus_{i=1}^k W(X_i)$ is the composition of the $(p_r^{(i)})$.

\begin{definition} \rm
Suppose for each $i\in\{1,\ldots,k\}$ that $(p_{r}^{(i)})$ is a non-decreasing positive sequence.  The \emph{union} of the sequences $(p_{r}^{(i)})$, $i=1,...,k$, is the unique (up to permutation of equal entries) non-decreasing sequence containing the entries in each $(p_r^{(i)})$.
\end{definition}
If for $1 \leq i \leq k$, $(p_r^{(i)})$ is the sequence of weights for $W(X_i)$ then the sequence of weights for $\bigsqcup_{i=1}^k W(X_i)$ is the union of the $(p_r^{(i)})$.\\

The set of weights and the sequence of weights for $X(v_1,v_2)$ decompose as follows.  Trivially we have
\begin{equation}\label{Wdecomp1}
W(X(v_1,v_2)) = \bigsqcup_{s \in \itin(X(v_1,v_2))}W(\itin^{-1}(s))
\end{equation}
Fix $s=s_1...s_m$ and let $J_1(s) = \{i\in \{1,...,m-1\}:s_i\leftrightarrow s_{i+1}\}$ and $J_2(s) = \{1,...,m-1\}\setminus J_1(s)$.  Let
\begin{equation*}
W_1(s) = \bigoplus_{i \in J_1(s)}W(X(s_i,s_{i+1};[s_i]))
\end{equation*}
and let
\begin{equation*}
W_2(s) = \bigoplus_{i \in J_2(s)}W(E(s_i,s_{i+1}))
\end{equation*}
Each $X(s_i,s_{i+1};[s_i])$ is the set of paths from vertex $s_i$ to vertex $s_{i+1}$ on the strongly connected graph with vertcies $[s_i]$, and each $E(s_i,s_{i+1})$ is a path consisting of a single edge from vertex $s_i$ to vertex $s_{i+1}$.  Lemma \ref{itin} implies that $W(\itin^{-1}(s))$ is given by
\begin{equation}\label{Wdecomp2}
W(\itin^{-1}(s)) = \bigsqcup_{w_2 \in W_2(s)}\{w_1+w_2:w_1 \in W_1(s)\}
\end{equation}
Therefore the sequence of weights for paths with itinerary $s$ is the union of translates of compositions of sequences of weights on s.c.c.'s (note that $W_2$ is a finite set, since the edge set is assumed finite).  Then, using \eqref{Wdecomp1}, the sequence of weights for $X(v_1,v_2)$ is the union, over admissible itineraries $s$, of the sequence of weights for paths with itinerary $s$.  \\

To find the asymptotic behaviour of the s.o.w. for $X(v_1,v_2)$ on an arbitrary graph, it thus suffices to find the asymptotic behaviour of the s.o.w. for sets of paths from one fixed vertex to another fixed vertex on a strongly connected graph, and to describe the effect of union and composition on the asymptotic behaviour.  The first point is addressed in Section 3, and the second point in Section 4.  It can be seen that translation will have no effect on the asymptotics.

\section{Strongly connected case}
In this section we compute the asymptotic behaviour for the s.o.w. of $X(v_1,v_2)$ on a strongly connected graph.  The main result of this section is Theorem \ref{irredGraph}.

\subsection{Linear Algebra Preliminaries}\label{secprelim}
First it is convenient to have $|E(u,v)|\leq 1$ for each $u,v \in V$, so that each edge $e \in E$ can be identified with the vertices $I(e)$ and $O(e)$.  Any graph can be converted into a graph that satisfies this condition, and whose paths and path weights are identical to those on the original path.  One way to do this is as follows:  if $|E(u,v)|>1$ then for each $e \in E(u,v)$ replace $e$ with a pair of edges $e_1,e_2$ and a vertex $v_1$ that satisfy $I(e_1) = u$, $O(e_1) = I(e_2) = v_1$ and $O(e_2) = v$, and $W(e_1)=W(e_2)=W(e)/2$.\\

A graph that satisfies $|E(u,v)|\in \{0,1\}$ for every pair of vertices $u$ and $v$ is labeled as follows.  If the graph has $n$ vertices, then label the vertices $1,...,n$, and for $1 \leq i,j \leq n$, if there is an edge from $j$ to $i$ then label it $e_{ij}$, and label the weights of edges as $w_{ij} = W(e_{ij})$.  In this section, a graph $G$ refers to a directed weighted graph with a labeling of the type just described.\\

Define the adjacency matrix $A$ to have entries $a_{ij}$ equal to $1$ if there is an edge from $j$ to $i$, and equal to $0$ otherwise.  To each directed graph with a labeling of the type described above, there corresponds an adjacency matrix.  Conversely, each adjacency matrix describes a directed graph.\\

Let $M$ be an $n\times n$ matrix.  Then $M$ is \emph{non-negative} or $M \geq 0$ if $m_{ij}\geq 0$ for $1\leq i,j \leq n$ and $M$ is \emph{positive} or $M>0$ if $m_{ij}>0$ for $1 \leq i,j \leq n$; the same definitions apply to a vector, treated as an $n\times 1$ matrix.  Also, $M$ is \emph{irreducible} if for each pair $(i,j)$ there is a positive integer $k$ such that $m_{ij}^{(k)}$, the $(i,j)$ entry of the matrix $M^k$, is non-zero.\\

If $A$ is an adjacency matrix and $a_{ij}^{(k)}$ is the $(i,j)^{th}$ entry of $A^k$, then $a_{ij}^{(k)}\neq 0$ if and only if there is a path from $j$ to $i$ of length equal to $k$.  Thus, a graph is strongly connected if and only if its adjacency matrix is irreducible.\\

For a graph $G$ with adjacency matrix $A$ the \emph{period} of $A$ is the positive integer $d = \gcd\{l(x):x \in X(v,v), v \in V\}$.  For $u$ in $V$, and $d$ the period of the adjacency matrix $A$, let $\overline{u}$ denote the set $\{v \in V: \exists x \in X(u,v),\,d|l(x)\}$, that is, the set of vertices $v$ such that there is a path from $u$ to $v$ of length equal to a multiple of the period.  If $A$ is irreducible it can be verified that $\{\overline{u}:u\in V\}$ is an equivalence relation, and so it partitions $V$.\\

On a strongly connected graph, $A^d$ admits a natural decomposition.  Label the vertex set $V = \{1,...,n\}$, and let $e_i$ be the $i^{th}$ standard basis vector in $\complex^n$.  If $j,i \in V$ and $\overline{j}\neq\overline{i}$ then there is no path from $j$ to $i$ of length a multiple of $d$, and vice-versa.  Since $a_{ij}^{(d)}=0 \Leftrightarrow$ there is no path from $j$ to $i$ of length $d$, it follows that $\spn\{e_i: i \in \overline{u}\}$ reduces $A^d$.  This fact is used in Corollary \ref{restr}.\\

Let $\sigma(M)$ denote the set of eigenvalues for $M$ and let $\rho(M) = \max \{|\lambda|:\lambda \in \sigma(M)\}$ denote the spectral radius.  The following is a well-known theorem for non-negative matrices which is proved, for example, in \cite{nnmtcs}.

\begin{theorem}[Perron-Frobenius]Let $M\geq 0$ be irreducible, and let $d$ be its period.  Then
\begin{remunerate}
\item $\rho(M)\in \sigma(M)$ and $\rho(M)$ has a one-dimensional eigenspace,
\item $\{e^{2\pi i k/d}\rho(M):k \in \mathbb{N}\} = \{\lambda \in \sigma(M):|\lambda|=\rho(M)\}$,
\item $M$ has a unique non-negative eigenvector $w$,
\item $w>0$ and satisfies $Mw = \rho(M)w$
\end{remunerate}
\end{theorem}
Observe that if $M$ is non-negative and irreducible with period $d$, then so is $M^{\top}$, so the above theorem can be translated for left eigenvectors.  A non-negative and irreducible matrix is \emph{primitive} if its period is equal to 1.  The following two results are used to obtain a simple proof of Lemma \ref{paths}.
\begin{theorem}\label{geo}
Let $B$ be a primitive matrix with $r = \rho(B)$ and let $w,u$ be non-negative, non-zero vectors.  Then, $w^{\top}(B/r)^m u$ converges geometrically to a positive constant, i.e., $\lim_{m\to\infty}w^{\top}(B/r)^m u$ exists and is positive, and $|w^{\top}(B/r)^m u - \lim_{m\to\infty}w^{\top}(B/r)^m u|= O(\nu^m)$ for some positive constant $\nu<1$.
\end{theorem}
\begin{proof}
In \cite{horn}, Theorem 8.5.1., it is proved that $(B/r)^m$ converges geometrically to a positive matrix.  Since a bounded linear mapping preserves geometric convergence, it follows that $w^{\top}(B/r)^mu$ converges geometrically.  Since $w,u \geq 0$, $w,u\neq 0$, and $(B/r)^m$ converges to a positive matrix, it follows that $\lim_{m\to\infty}w^{\top}(B/r)^mu$ is positive.
\end{proof}

\begin{definition} \rm\label{ubar}
Let $G$ be a graph with vertices $V = \{1,...,n\}$, let $d$ be a positive integer and let $U$ be a subset of $V$ with $|U|=k>0$.  Let $A = (a_{ij})_{1\leq i,j \leq n}$ be an $n\times n$ matrix and let $w = (w_i)_{i=1,...,n}$ be an $n\times 1$ vector.  The restriction of $A$ to $U$, denoted $A|_{U}$, is the $k\times k$ matrix $(a_{ij})_{(i,j) \in U\times U}$ and the restriction of $w$ to $U$, denoted $w|_{U}$, is the $k\times 1$ vector $(w_i)_{i \in U}$.
\end{definition}

\begin{corollary}\label{restr}
Let $G$ be a strongly connected graph for which the adjacency matrix $A$ has period $d$.  Let $u$ be any vertex on $G$ and let $B = A^d|_{\overline{u}}$, then $B$ is a primitive matrix.  If $w$ denotes the positive eigenvector for $A$, then the restriction $w|_{\overline{u}}$ is the unique positive eigenvector for $B$ and $\rho(B) = \rho(A)^d$.
\end{corollary}
\begin{proof}
Since $A\geq 0$, $B\geq 0$.  Also, $B$ is irreducible, since for every $v \in \overline{u}$, $v \neq u$, the fact that $G$ is strongly connected implies that there is path from $v$ to $u$, and since $A$ has period $d$, and since by assumption there is a path from $u$ to $v$ whose length is a multiple of $d$, it follows that there is a path from $v$ to $u$ whose length is a multiple of $d$.  Since $\spn\{e_i:  i \in \overline{u}\}$ reduces $A^d$ as mentioned earlier, it follows that each eigenvector of $B$ is the restriction to $\overline{u}$ of an eigenvector of $A^d$, and in particular $\sigma(B) \subset \{\lambda^d: \lambda \in\sigma(A)\}$.  In particular, if $\lambda \in \sigma(B)$ and $|\lambda| = \rho(B)$, then $\lambda = \rho(B)$.  Since $B$ satisfies the hypotheses of the Perron-Frobenius theorem, it follows that $B$ must have period 1, i.e., $B$ is primitive.  The rest of the corollary follows from the above observations.
\end{proof}

\subsection{Graph Approximation}
A weighted graph is \emph{uniformly weighted} if each edge has the same weight assigned to it.  In this section, for an arbitrary directed weighted graph $G$ we construct uniformly weighted graphs containing the relevant structure of $G$.\\

An \emph{approximation base} $b \in \reals ^+$ is \emph{admissible} if $b<\min w_{ij}$.  For $b$ admissible, define the \emph{approximate graph} $G(b)$ of the graph $G$ as follows.  For each $e_{ij} \in E$ let $C_{ij} = \left\lfloor w_{ij} / b \right\rfloor$ and replace $e_{ij}$ with a chain of $C_{ij}$ edges and $C_{ij}-1$ vertices.  Assign the weight $b$ to each edge on $G(b)$, so that $G(b)$ is a uniformly weighted graph.  Note that paths on $G$ are in one to one correspondence with paths on $G(b)$ that start and end on vertices corresponding to vertices in $G$.  Thus for a path $x$ on $G$, we let $x_b$ denote the corresponding path on $G(b)$.  Also note that $G$ strongly connected $\Leftrightarrow G(b)$ strongly connected and that $G$ not a cycle $\Leftrightarrow G(b)$ not a cycle.  For an example of an approximate graph see Figure \ref{figapproxgraph}.\\

\begin{figure}
\centering
\mbox{\subfigure[A weighted directed graph, with edge weights indicated.]{\includegraphics[height=60mm,width=75mm]{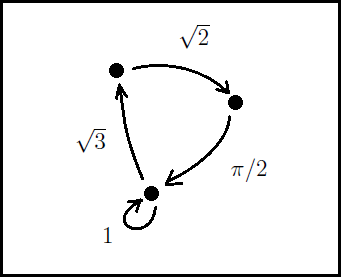}}\quad
\subfigure[Approximate graph:  each edge is assigned the weight $1/2$.]{\includegraphics[height=60mm,width=75mm]{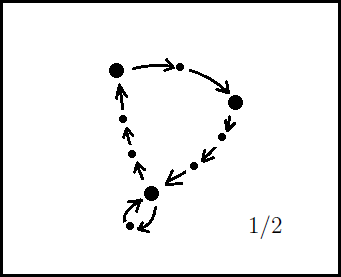} }}
\caption{An example of a graph and an approximate graph with $b=1/2$.}
\label{figapproxgraph}
\end{figure}

If $\{w_{ij}:1\leq i,j\leq n\} \subset \{ kb: k \in\mathbb{N}\}$ then the weights of paths on $G$ and on $G(b)$ correspond exactly.  More generally, the weights of paths on the approximate graph are close to the weights of the corresponding paths on the original graph.  This is expressed more precisely in the following lemma.

\begin{lemma}\label{uniform}
For each $\epsilon>0$ there exists $b\in\reals^+$ such that $|W(x_b)/W(x)-1|<\epsilon$ for all paths $x$ on $G$ and corresponding paths $x_b$ on $G(b)$.
\end{lemma}
\begin{proof}
To each edge $e_{ij}$ on $G$ there corresponds a path $x_{ij}$ on $G(b)$, and $W(x_{ij}) = C_{ij}b = \left\lfloor w_{ij}/b \right\rfloor b$.  As $b \rightarrow 0^+$, $\delta = \max\{|w_{ij} - W(x_{ij})|\} \rightarrow 0$.  Let $w_{min} = \min w_{ij}>0$.  Then for any $x$, $x_b$,
\begin{equation*}
|W(x)- W(x_b)| \leq \delta l(x) \leq \delta \frac{W(x)}{w_{min}}
\end{equation*}
which gives
\begin{equation*}
|W(x_b)/W(x) - 1| \leq \delta/w_{min}
\end{equation*}
For $\epsilon>0$, taking $b$ small enough so that $\delta<w_{min}\epsilon$ gives the desired result.
\end{proof}

As $b$ approaches zero the approximate graph $G(b)$ becomes very large.  The following result is useful in relating the behaviour of $G(b)$ to the original graph.

\begin{lemma}\label{collapse}
Let $G$ be strongly connected and let $G(b)$ be an approximate graph, and let $A = (a_{ij})$ and $A(b)$ be their respective adjacency matrices.  Let $\lambda = \rho(A(b))$, then the matrix $B$ whose entries are given by $a_{ij}\lambda^{-C_{ij}}$ has $\rho(B)=1$.
\end{lemma}
\begin{proof}
In the next section it is shown that the adjacency matrix of a strongly connected graph has spectral radius $\lambda\geq 1$.  In particular, $\lambda \neq 0$, and the matrix $A(b)/\lambda$ has $\rho(A(b)/\lambda)=1$.  Since $A(b)/\lambda$ is non-negative and irreducible, by the Perron-Frobenius theorem
$\frac{A(b)}{\lambda}$ has a positive left eigenvector $v^{\top}$ such that $v^{\top}A(b)/\lambda = v^{\top}$.  Moreover, the restriction of $v^{\top}$ to the vertices of $G$ is a positive vector that satisfies $v^{\top}B = v^{\top}$.  This is because a unit vector corresponding to vertex $j$, in being set to vertex $i$ by the application of $(A(b)/\lambda)^{C_{ij}}$, is multiplied by a factor $\lambda^{-C_{ij}}$ .  Since $B$ is irreducible and non-negative, by the Perron-Frobenius theorem $v^{\top}$ is the unique positive left eigenvector for $B$, and the eigenvalue corresponding to $v^{\top}$ is equal to $\rho(B)$.  Therefore $\rho(B)=1$.
\end{proof}

\subsection{Weight Distribution}
In this section, we consider a strongly connected graph $G$, and for arbitrary vertices $v_1$ and $v_2$ not necessarily distinct, we determine the asymptotic behaviour for the s.o.w. of $X(v_1,v_2)$, the set of paths from $v_1$ to $v_2$.  Recall that a \emph{cycle} is a strongly connected graph in which each vertex has in-degree and out-degree both equal to 1.

\begin{lemma}
Let $G$ be a strongly connected graph and let $A$ be its adjacency matrix.  Then the spectral radius $\rho(A)\geq 1$, and $\rho(A)=1$ if and only if $G$ is a cycle.
\end{lemma}
\begin{proof}
If $G$ is strongly connected then in particular, for each vertex $v$ there is a vertex $u$ such that $v \rightarrow u$.  For a vector $w = (w_1,...,w_n)^{\top}$, define $\|w\|_1 = \sum_{i=1}^n |w_i|$, then if $w \geq 0$, $\|Aw\|_1 \geq\|w\|_1$.  Since $A$ is non-negative and irreducible, the Perron-Frobenius theorem applies, and there exists a unique positive eigenvector $w$ whose eigenvalue is equal to the spectral radius $\rho(A)$, and it follows from the above observation that the eigenvalue for $w$ must be $\geq 1$.  If for some $v \in V$ there exist $u_1\neq u_2$ such that $v \rightarrow u_1$ and $v \rightarrow u_2$ then $\|Aw\|_1 >\|w\|_1$.  For a strongly connected graph this is only possible if $G$ is not a cycle.
\end{proof}

If $G$ is a cycle the s.o.w. is easily described.

\begin{lemma}\label{cycle}Let $G$ be a cycle, and let $w_0$ be the weight of any path that goes exactly once around the cycle, called the \emph{cycle weight}.  Then if $(p_r)$ denotes the sequence of weights for $X(v_1,v_2)$,
\begin{equation*}
\lim_{r\to\infty} p_r/r=w_0
\end{equation*}
\end{lemma}
\begin{proof}
Let $w_{21}$ be the weight of the shortest path from $v_1$ to $v_2$.  Then the sequence of weights $(p_r)$ is given by $p_r = w_{21} + w_0r$, $r = 0,1,2,...$.  In particular, $\lim_{r \to\infty} p_r/r = w_0$, the cycle weight.
\end{proof}

If $G$ is not a cycle, first we consider the case of a uniformly weighted graph, for which the asymptotic behaviour of the s.o.w. is more easily computed.

\begin{lemma}\label{paths}
Suppose $G$ is a strongly connected uniformly weighted graph which is not a cycle, and let $b>0$ be the weight of each edge.  Let $A$ be the adjacency matrix for $G$, and let $\lambda=\rho(A)$ be its spectral radius and $d>0$ its period.  For vertices $v_1,v_2 \in V$ not necessarily distinct, let $X$ be the set of paths from $v_1$ to $v_2$ on $G$ and let $(p_r)$ be the sequence of weights for $X$.  Then
\begin{equation*}
\lim_{r\to\infty}p_r/\log r = b/\log \lambda
\end{equation*}
\end{lemma}

{\em Proof.}
For a path $x$ on $G$, $W(x)=l(x)b$.  There is an integer $i \in \{0,1,...,d-1\}$ such that each path from $v_1$ to $v_2$ has length $md+i$ for some integer $m$.  As defined in Section \ref{secprelim}, let $\overline{v_2}$ denote the set of vertices $y$ for which there is a path either from $y$ to $v_2$ or from $v_2$ to $y$ whose length is a multiple of $d$.  Let $u=(0,0,...,0,1,0,...)$ be the vector equal to 1 in the $v_1$ entry and zero elsewhere.  As in Definition \ref{ubar} let $B = A^d|_{\overline{v_2}}$ and redefine $u$ to be the restriction $A^iu|_{\overline{v_2}}$.  Let $w=(0,0,...,0,1,0,...)$ on $\overline{v_2}$ be the vector equal to 1 in the $v_2$ entry, and zero elsewhere.  Then $c_m = w^{\top}B^mu$ counts the number of paths from $v_1$ to $v_2$ of length $md+i$.  Using Corollary \ref{restr} and Theorem \ref{geo}, $\rho(B) = \lambda^d$ and $c_m\lambda^{-dm}$ converges geometrically to some positive constant $c$.  In other words, $c_m = c\lambda^{md}(1+R(m))$ with $|R(m)|\leq R\nu^m$ for some $R>0$ and $0<\nu<1$, so that
\begin{equation}\label{rest}
\left |\sum_{j=0}^m c_j - c\sum_{j=0}^m \lambda^{jd}\right | \leq R\sum_{j=0}^m (\lambda^j\nu)^d = o(\lambda^{md})
\end{equation}\\
For each $r$, $p_r = (md+i)b$ for some $m$ and some $\Delta$ that satisfy
\begin{equation*}
r = \sum_{j=0}^{m-1} c_j + \Delta
\end{equation*}
and $\Delta \in \{1,...,c_m\}$.  Using \eqref{rest},
\begin{eqnarray*}
r &=& c\frac{\lambda^{(m+1)d}-1}{\lambda^d-1} + o(\lambda^{md})+ \Delta\\
&=& \lambda^{md}(c\frac{\lambda^d-\lambda^{-md}}{\lambda^d-1} + o(1)+ \lambda^{-md}\Delta)
\end{eqnarray*}
Taking logs,
\begin{equation}\label{thelogr}
\log r = md\log \lambda + \log(c\frac{\lambda^d-\lambda^{-md}}{\lambda^d-1} + o(1)+ \lambda^{-md}\Delta)
\end{equation}
and note that the argument to the second $\log$ is both upper- and lower-bounded by positive numbers, i.e.,
\begin{equation*}
\liminf_{m\to\infty}c\frac{\lambda^d-\lambda^{-md}}{\lambda^d-1} + o(1) + \lambda^{-md}\Delta \geq c\frac{\lambda^d}{\lambda^d-1}>0
\end{equation*}
and
\begin{equation*}
\limsup_{m\to\infty}c\frac{\lambda^d-\lambda^{-md}}{\lambda^d-1} + o(1)+ \lambda^{-md}\Delta \leq c\frac{\lambda^d}{\lambda^d-1} + \limsup_{m\to\infty}\lambda^{-md}\Delta<\infty
\end{equation*}
since $o(1)\rightarrow 0$ and $\Delta\geq 0$, but $\Delta = O(\lambda^{md})$.  Using this observation and \eqref{thelogr},
\begin{equation*}
\log r / p_r = \log r / ((md+i)b) = \log\lambda / b + O(1/m)
\end{equation*}
and so
\begin{equation*}
\lim_{r\to\infty}p_r/\log r = b/\log \lambda \in \reals ^+ \cvd
\end{equation*}

For a general strongly connected graph we get a similar result after taking a limit of approximate graphs.
\begin{theorem}\label{irredGraph}
Suppose $G$ is strongly connected and is not a cycle.  For $v_1,v_2\in V$ not necessarily distinct, let $X$ be the set of paths from $v_1$ to $v_2$ on $G$ and let $(p_r)$ be the sequence of weights for $X$.  Then
\begin{equation*}
\lim_{r\to\infty}p_r/\log r = s
\end{equation*}
where $s>0$ is such that the matrix $B$ with entries $a_{ij}e^{-s^{-1}w_{ij}}$ has $\rho(B)=1$.
\end{theorem}

In other words, the limit $s$ can be understood as an exponential decay constant along edges such that the resulting matrix is stochastic.
\begin{remark}
If $P = (p_{ij})$ is a stochastic matrix then $\rho(P) = 1$.  Using $a_{ij} = 1$ if $p_{ij}>0$ and $w_{ij} = -\log p_{ij}$ gives $s = 1$ in this case.  Plotting $(p_r)$ versus $r$, with $(p_r)$ the \emph{probabilities} of paths in decreasing order gives a graph asymptotic to a straight line of slope $-1$, on a log-log plot.  A sub-stochastic matrix, i.e., $\rho(P)<1$ will have $s<-1$, gives a graph asymptotic to a line of slope $<-1$.
\end{remark}

Although, as shown in Lemma \ref{uniform}, the weights of paths on $G(b)$ are close to the weights of paths on $G$, in the sequence of weights they may show up in the wrong order, i.e., for paths $x,y$ we may have $W(x)<W(y)$ but $W(x_b)>W(y_b)$.  Nevertheless, the asymptotics are related, as shown in the following.
\begin{lemma}\label{pigeon}
Let $s \in  \reals ^+$ and let $(f_n)$ be a non-decreasing positive sequence with $f_n/\log n \rightarrow s$ as $n \rightarrow \infty$.  For $\epsilon \in (0,1)$, let $(h_n)$ be a positive sequence that satisfies $|h_n/f_n-1|<\epsilon$ uniformly in $n$.  Let $\sigma:\mathbb{N}\rightarrow\mathbb{N}$ be a permutation of $\mathbb{N}$ chosen so that the sequence $(g_n)$ defined by
\begin{equation*}
g_n = h_{\sigma(n)}
\end{equation*}
is non-decreasing.  Then, $\limsup |g_n/\log n-s|\leq s\epsilon$.
\end{lemma}

\begin{proof}
Let $(\epsilon_n)$ be a sequence with $|\epsilon_n|<\epsilon$ uniformly in $n$ and such that
\begin{equation*}
h_n = (1+\epsilon_n)f_n
\end{equation*}
for each $n$.  Let $\delta>0$, and let $\Delta=\delta+(s+\delta)\epsilon$.  There exists $N\in\mathbb{N}$ so that for $n\geq N$, $|f_n/\log n-s|<\delta$, and this gives $|h_n/\log n - s|<\Delta$.   Suppose for some $j\geq N$ that $g_j/\log j < s - \Delta$.  Then for $m\leq j$ we must have $\sigma(m) < j$.  But then, $\sigma$ maps $\{1,...,j\}$ injectively into $\{1,...,j-1\}$, which is impossible.\\
Now, take $N'\geq N$ big enough so that $(s+\Delta)\log N'>\max_{n<N}h_n$ and suppose for some $j\geq N'$ that $g_j/\log j> s+\Delta$.  Let $\tau$ denote the inverse of $\sigma$, so that $g_{\tau(n)} = h_n$ for all $n$.  Then for $m\leq j$ we must have $\tau(j)<j$.  But then, $\tau$ maps $\{1,...,j\}$ injectively onto $\{1,...,j-1\}$, which is again impossible.  Since $\delta$ is arbitrary, and $\Delta\rightarrow s\epsilon$ as $\delta \rightarrow 0$, it follows that $\limsup |g_n/\log n-s|\leq s\epsilon $.
\end{proof}

\noindent We now prove Theorem \ref{irredGraph}.
\begin{proof}
Using Lemma \ref{uniform}, for $0<\epsilon<1$ there exists $b>0$ s.t. $0<\alpha,\beta<b$ implies
\begin{equation*}
\frac{1-\epsilon}{1+\epsilon} \leq \frac{W(x_{\alpha})}{W(x_{\beta})} \leq \frac{1+\epsilon}{1-\epsilon}
\end{equation*}
for corresponding paths $x_{\alpha}$ and $x_{\beta}$.  Let $(p_r)$ denote the sequence of weights for $W(x_{\alpha})$, and let $s_{\alpha}$ denote $\lim_{r\rightarrow\infty} p_r/\log r$.  Define $s_{\beta}$ in the same way, and let $\delta = \max\{|1-\frac{1-\epsilon}{1+\epsilon}|,|1-\frac{1+\epsilon}{1-\epsilon}|\}$.  Applying Lemma \ref{pigeon} gives
\begin{equation*}
|s_{\alpha}-s_{\beta}|\leq s_{\alpha}\delta
\end{equation*}
If $s_{\alpha} = 0$ then $s_{\beta}=0$.  Otherwise, $1-\delta \leq s_{\beta}/s_{\alpha}\leq 1+\delta$.  Taking logs gives
\begin{equation*}
\log(1-\delta) \leq \log(s_{\alpha}) - \log(s_{\beta}) \leq \log (1+\delta)
\end{equation*}
As $b\rightarrow 0$, $\epsilon$ can be chosen so that $\epsilon\rightarrow 0$, and so $\delta\rightarrow 0$.  Therefore $(\log(s_b))$ is Cauchy and therefore converges, as $b\rightarrow 0$.  By continuity of the exponential function, $\lim_{b\to 0 }s_b$ exists; denote the limit by $s$.  For $\epsilon>0$ let $\alpha$ satisfy $|s_{\alpha}-s|<\epsilon$ and using Lemma \ref{uniform}, let it satisfy also $|W(x_{\alpha})/W(x)-1|<\epsilon$ uniformly for $x$ on $G$ and $x_{\alpha}$ on $G(\alpha)$.  Using Lemma \ref{pigeon},
\begin{equation*}
\limsup |p_r/\log r - s| \leq \limsup |p_r/\log r - s_{\alpha}| + |s_{\alpha} - s| < \epsilon (s_{\alpha} + 1)
\end{equation*}
Since $\epsilon$ is arbitrary $\lim_{r \to\infty} p_r/\log r$ exists and is equal to $s=\lim_{b\to 0}s_b$.  To obtain $s$, for $b$ admissible let $A(b)$ denote the adjacency matrix of $G(b)$, and let $\lambda_b = \rho(A(b))$.  Let $B(b)$ denote the matrix with entries $a_{ij}\lambda_b^{-C_{ij}}$, then by Lemma \ref{collapse} it follows that $\rho(B(b))=1$ for each $b$.  Using Lemma \ref{paths}, $s_b = b/\log \lambda_b$ for each $b$, so that $\lambda_b^{-C_{ij}} = e^{-s_b^{-1}C_{ij}b}$.  But $C_{ij}b \rightarrow w_{ij}$ and $s_b \rightarrow s$ as $b \rightarrow 0$, therefore $s$ is given by the condition $\rho(B)=1$, where $B$ is the matrix with entries $a_{ij}e^{-s^{-1}w_{ij}}$.
\end{proof}

\section{Composition, Union of Sets of Paths}
In this section we describe the effect of union and composition, as defined in Section 2, on the asymptotic behaviour of sequences of weights.
\begin{lemma}\label{comp}
For $1 \leq i \leq k$ let $(p_{r}^{(i)})$ be a non-decreasing positive sequence, and let $(c_r)$ be the composition of the $(p_{r}^{(i)})$.  Then,
\begin{remunerate}
\item if for each $i$, $\lim_{r\to\infty} p_r^{(i)}/r \in \reals ^+$ then
\begin{equation*}
\lim_{r\to\infty}c_r^k/r = k!\prod_{i=1}^k s_i
\end{equation*}
where $s_i = \lim_{r\to\infty}p_{r}^{(i)}/r$ for $1 \leq i\leq k$, and
\item if $\lim_{r\to\infty}p_{r}^{(i)}/\log r >0$ for all $i$ and is $<\infty$ for some $i$ then
\begin{equation*}
\lim_{r\to\infty}c_r/\log r = \min s_i
\end{equation*}
where $s_i = \lim_{r\to\infty}p_{r}^{(i)}/\log r$ for $1 \leq i \leq k$.
\end{remunerate}
\end{lemma}
\begin{proof}
First we prove part 1.  For $1 \leq d < k$ let $(q_r)$ denote the composition of $(p_{r}^{(i)})$, $i=1,...,d$, and suppose that $s_q = \lim_{r\to\infty}q_r^d/r$ exists and $= d!\prod_{i=1}^d s_i$.  Let $(p_r)$ denote $(p_r^{(d+1)})$ and $s_p$ denote $s_{d+1}$.  Then there exist positive functions $\delta:\mathbb{N}\rightarrow\reals^+$ and $\gamma:\mathbb{N}\rightarrow\reals^+$ which are non-increasing and satisfy
\begin{equation}\label{dg1}
|i/p_i - s_p^{-1}|<\delta(i),\, |q_j/j^{1/d} - s_q^{1/d}|<\gamma(j)
\end{equation}
and $\delta(i),\gamma(j) \rightarrow 0$ as $i,j \rightarrow \infty$.\\

Let $(c_r)$ be the composition of $(p_r^{(i)})$, $i=1,...,d+1$, equivalently, the composition of $(p_r)$ and $(q_r)$.  For $c \in\reals^+$ define $r(c)$ to be the number of entries in $(c_r)$ which are $\leq c$, i.e.,
\begin{equation*}
r(c) = \#\{(i,j): p_i+q_j \leq c\}
\end{equation*}
We will estimate $r(c)$ in the limit of large $c$.  For $f,g\neq 0$, the following notation is used in what follows:
\begin{romannum}
\item $f(c)=O(g(c)) \Leftrightarrow \limsup_{c\to\infty}|f(c)/g(c)| < \infty$
\item $f(c)=o(g(c)) \Leftrightarrow \lim_{c\to\infty}f(c)/g(c)=0$
\item $f(c)\sim g(c) \Leftrightarrow \lim_{c\to\infty}f(c)/g(c)=1$
\end{romannum}
Observe that $r(c) = \sum_{j=1}^{J(c)}I(j,c)$ where $p_1+q_{J(c)} \leq c < p_1+q_{J(c)+1}$ and $p_{I(j,c)} \leq c - q_j < p_{I(j,c)+1}$.  Since $(p_r)$ and $(q_r)$ are non-decreasing, it follows that $J(c)$ is non-decreasing, and that $I(j,c)$ is non-increasing in $j$, and non-decreasing in $c$.  From $\eqref{dg1}$, $J(c)\sim s_q^{-1} c^d$, and for each $j$, $I(j,c)\sim s_p^{-1}c$.\\

For each $n$, $I(j,c)\geq n$ for most $j$, if $c$ is large enough.  More precisely, let $J_n(c) = \max \{j:I(j,c)\geq n\}$.  Since $I(j,c)$ is non-increasing in $j$, $I(j,c)<n$ if and only if $j>J_n(c)$.  Let $j = J_n(c)+1$ and let $J$ denote $J(c)$, then $I(j,c)<n$ and so $p_n + q_j \geq c$.  Since $c \geq p_1 + q_J$ it follows that $q_J-q_j \leq p_n - p_1$.  Let $\sigma = s_q^{1/d}$, then from \eqref{dg1} and since $\gamma(j)$ is non-increasing,
\begin{eqnarray*}
q_J - q_j &\geq& J^{1/d}(\sigma-\gamma(J)) - j^{1/d}(\sigma+\gamma(j))\\
&\geq& (J^{1/d}-j^{1/d})\sigma-2\gamma(j)J^{1/d}
\end{eqnarray*}
Since the function $f(x)=x^{1/d}$ is concave, $J^{1/d}-j^{1/d} \geq (J-j)\frac{d}{dx}(x^{1/d})\big{|}_{x=J} = \frac{1}{d}(J-j)J^{(1/d)-1}$, therefore
\begin{eqnarray*}
J-j &\leq& (d/\sigma)\cdot((q_J-q_j)J^{1-(1/d)} + 2\gamma(j)J)\\
&\leq& (d/\sigma)\cdot((p_n-p_1)J^{1-(1/d)}+2\gamma(j)J)\\
&=& C_1J^{1-(1/d)} + C_2\gamma(j)J
\end{eqnarray*}
where $C_1 = (d/\sigma)\cdot(p_n-p_1)$ and $C_2 = 2d/\sigma$.  Since $q_j \geq q_J-(p_n-p_1)$, $j\rightarrow\infty$ as $c\rightarrow\infty$, so that $\gamma(j)\rightarrow 0$, which implies that $(J(c)-J_n(c))/J(c) \rightarrow 0$ as $c\rightarrow\infty$, i.e., $J(c)-J_n(c) = o(J(c))$, justifying the statement ``$I(j,c) \geq n$ for most $j$, if $c$ is large enough''.\\

Define $s= \min\{s_p^{-1},s_q^{1/d},1\}$ and for $\epsilon>0$, $\epsilon<s$ let $N,M \in \mathbb{N}$ be such that $i\geq N$ implies $\delta(i) < \epsilon $ and $j\geq M$ implies $\gamma(j) < \epsilon$, and note that $J_N(c)> M$ for $c$ large enough.  Since $I(j,c) = O(c)$ for each $j$, it follows that $\sum_{j=1}^M I(j,c) = O(c)$.  Since $I(j,c)<N$ when $j>J_N(c)$ and $J(c)-J_N(c) = o(J(c))$, it follows that $\sum_{j=J_N(c)+1}^{J(c)}I(j,c) = o(J(c))$.  Since $O(c) = o(c^{d+1})$ and $o(J(c)) = o(c^{d+1})$,
\begin{equation}\label{rsum0}
r(c) = \sum_{j=M+1}^{J_k(c)} I(j,c) + o(c^{d+1})
\end{equation}
Using \eqref{dg1}, $I(j,c)$ satisfies
\begin{equation*}
(s_p^{-1}-\delta(I(j,c)+1))p_{I(j,c)+1}-1 < I(j,c) < (s_p^{-1} + \delta(I(j,c)))p_{I(j,c)}
\end{equation*}
Since $\delta$ is non-increasing, $\delta(I(j,c)+1)$ can be replaced with $\delta(I(j,c))$ on the left-hand side.  Using $p_{I(j,c)} \leq c - q_j < p_{I(j,c)+1}$ and \eqref{dg1} again gives
\begin{equation*}
(s_p^{-1}-\delta)(c-(s_q^{1/d}+\gamma)j^{1/d})-1 < I < (s_p^{-1} + \delta)(c-(s_q^{1/d}-\gamma)j^{1/d})
\end{equation*}
or
\begin{equation}\label{Ibnd}
|I- (c/s_p - s_q^{1/d}s_p^{-1}j^{1/d})| \leq 1 + \delta c + (\delta s_q^{1/d} + \gamma s_p^{-1} + \delta\gamma)j^{1/d}
\end{equation}
where the arguments to $\delta$, $\gamma$ and $I$ have been suppressed.  Using $\delta<\epsilon<1$ and $\gamma<\epsilon<1$ gives
\begin{equation}\label{Ebnd}
\sum_{j=M+1}^{J_N(c)}\delta c + (\delta s_q^{1/d} +\gamma s_p^{-1} +\delta\gamma)j^{1/d} \leq \epsilon J(c) \left [c + (s_q^{1/d}+s_p^{-1}+1)J(c)^{1/d}\right ]
\end{equation}
Let $E(c) = J(c)\cdot(c + (s_q^{1/d}+s_p^{-1}+1)J(c)^{1/d})$, then since $J(c) \sim s_q^{-1}c^d$, $E(c) \sim Kc^{d+1}$ where $K = s_q^{-1}(1 + (s_q^{1/d}+s_p^{-1}+1)s_q^{-1/d})$ is a constant.  With this observation, and using \eqref{Ibnd} and \eqref{Ebnd},
\begin{equation}\label{Ibnd2}
\sum_{j=M+1}^{J_N(c)}|I- (c/s_p - s_q^{1/d}s_p^{-1}j^{1/d})| \leq J(c) + \epsilon Kc^{d+1} + o(c^{d+1})
\end{equation}
We now estimate the term on the left-hand side of \eqref{Ibnd2}.  Using $J_N(c) = J(c)\cdot(1+o(1))$, $\sum_{j=M+1}^{J_N(c)}j^{1/d} = \frac{d}{d+1}J_N(c)^{1+(1/d)} + O(1)$ and $J(c) = c^d/s_q + o(c^d)$ gives
\begin{eqnarray}\label{LHbnd}\nonumber
\sum_{j=M+1}^{J_N(c)} (c/s_p - s_q^{1/d}s_p^{-1}j^{1/d}) &=& \frac{c^{d+1}}{s_qs_p} - \frac{d}{d+1}\frac{s_q^{1/d}}{s_p}s_q^{-(1+(1/d))}c^{d+1} + o(c^{d+1})\\ \nonumber
&=& \frac{c^{d+1}}{s_qs_p}(1-\frac{d}{d+1}) + o(c^{d+1})\\
&=& \frac{c^{d+1}}{(d+1)s_qs_p} + o(c^{d+1})
\end{eqnarray}
Using \eqref{rsum0}, \eqref{Ibnd2}, \eqref{LHbnd} and $J(c) = o(c^{d+1})$,
\begin{equation*}
|r(c)-\frac{c^{d+1}}{(d+1)s_qs_p}| = \epsilon Kc^{d+1} + o(c^{d+1})
\end{equation*}
or
\begin{equation*}
\limsup_{c\to\infty}\left|\frac{r(c)}{c^{d+1}} - \frac{1}{(d+1)s_qs_p}\right| \leq \epsilon K
\end{equation*}
Since $\epsilon>0$ is arbitrary it follows that $\lim_{c\to\infty} r(c)/c^{d+1} = 1/((d+1)s_qs_p)$.  Substituting and inverting,
\begin{equation*}
\lim_{r\to\infty} \frac{c^{d+1}}{r} = (d+1)!\prod_{i=1}^{d+1}s_i
\end{equation*}
For a set $(p_{r}^{(i)})$ of sequences, $i=1,...,k$, applying this rule $k-1$ times gives part 1.\\

We now prove part 2.  For sequences $(p_r)$ and $(q_r)$ consider now $s_p = \lim_{r\to\infty}p_r/\log r$ and $s_q = \lim_{r\to\infty}q_r/\log r$, and suppose without loss of generality that $s_p < s_q$.  Let $\delta(i)$ and $\gamma(j)$ be non-increasing positive functions of the indices $i$ and $j$ that satisfy
\begin{equation}\label{dg2}
|\log i/p_i - s_p^{-1}|<\delta(i),\, |q_j/\log j - s_q|<\gamma(j),\, |\log j/q_j - s_q^{-1}|<\gamma(j)
\end{equation}
and $\delta(i)\rightarrow 0$ as $i\rightarrow\infty$, $\gamma(j) \rightarrow 0$ as $j \rightarrow \infty$.\\

Let $(c_r)$ be the composition of $(p_i)$ and $(q_j)$ and let $r(c)$, $I(j,c)$, $J(c)$ and $J_k(c)$ be defined as before.  As before, $J(c)$ is non-decreasing, and $I(j,c)$ is non-increasing in $j$, and non-decreasing in $c$.  From $\eqref{dg2}$, $J(c)\sim e^{s_q^{-1}c}$, and for each $j$, $I(j,c)\sim e^{s_p^{-1}c}$.  For $0<\epsilon<\min \{s_p^{-1},s_q\}$ take $N,M \in \mathbb{N}$ so that $i \geq N$ and $j\geq M$ implies $\delta(i)<\epsilon$, $\gamma(j)<\epsilon$, and
\begin{equation}\label{bnd1}
\delta(i) s_q + \gamma(j) (s_p^{-1} + \delta(i)) < \epsilon
\end{equation}
Since for each $j$, $I(j,c)\rightarrow \infty$ as $c\rightarrow\infty$, let $c$ be large enough that $J_N(c)> M$.  Then
\begin{equation}\label{rsum1}
r(c) = \sum_{j=1}^M I(j,c) + \sum_{j=M+1}^{J_N(c)} I(j,c) + \sum_{j = J_N(c)+1}^{J(c)} I(j,c)
\end{equation}
If $j\leq J_N(c)$ then $I(j,c)\geq N$ and so $\delta(I(j,c)) < \epsilon$, so that $I(j,c) \leq e^{(s_p^{-1}+\epsilon)c}$, using \eqref{dg2} and the fact that $p_{I(j,c)}\leq c$.  Therefore, the first sum is $\leq Me^{(s_p^{-1}+\epsilon)c}$.  Since $I(j,c)<N$ for $j>J_N(c)$, the third sum is $\leq NJ(c)$, which is $\leq Ne^{(s_q^{-1}+\epsilon)c}$, using \eqref{dg2} and the fact that $J(c)>M$, and $\gamma(j)<\epsilon$ when $j\geq M$.  For $M+1\leq j \leq J_N(c)$, $I(j,c)$ satisfies
\begin{equation*}
\exp((s_p^{-1}-\delta)(c-(s_q+\gamma)\log j)-1) < I < \exp((s_p^{-1} + \delta)(c-(s_q-\gamma)\log j))
\end{equation*}
where $\delta = \delta(I(j,c))$ and $\gamma = \gamma(j)$.  Using $\delta<\epsilon$, $\gamma<\epsilon$, and \eqref{bnd1} gives
\begin{equation*}
\exp((s_p^{-1}-\epsilon)c - 1)j^{- s_p^{-1}s_q - \epsilon} < I < \exp(s_p^{-1}c + \epsilon c)j^{- s_p^{-1}s_q + \epsilon}
\end{equation*}
If $\epsilon<s_p^{-1}s_q-1$, which is true for $\epsilon$ small enough, then $\sum_{j=0}^{\infty} j^{-s_p^{-1}s_q+\epsilon}<\infty$.  Then, from \eqref{rsum1} and from the above observations,
\begin{equation*}
e^{(s_p^{-1}-\epsilon)c - 1}(M+1)^{- s_p^{-1}s_q - \epsilon} \leq r(c) \leq Me^{(s_p^{-1}+\epsilon)c} +  Ce^{(s_p^{-1} + \epsilon)c} + Ne^{(s_q^{-1}+\epsilon)c}
\end{equation*}
where $C =  \sum_{j=0}^{\infty} j^{-s_p^{-1} s_q+\epsilon}$, and the lower bound is obtained by taking the $j=M+1$ term of the second sum in \eqref{rsum1}.  Using $s_q^{-1}<s_p^{-1}$, it follows that
\begin{equation*}
\limsup_{c\to\infty}|\log r(c)/c-s_p^{-1}| \leq \epsilon
\end{equation*}
Since $\epsilon>0$ is arbitrary, $\lim_{c\to\infty}\log r(c)/c = s_p^{-1}$, or inverting,
\begin{equation*}
\lim_{r\to\infty}c_r/\log r = s_p
\end{equation*}
For $1\leq d < k$, if $(q_r)$ is the composition of $(p_r^{(i)})$, $i=1,...,d$ and $(p_r) = (p_r^{(d+1)})$, then applying the rule $k-1$ times gives part 2.
\end{proof}

\begin{lemma}\label{union}
Let $(p_r)$ and $(q_r)$ be non-decreasing positive sequences such that either $\lim_{r\to\infty}p_r^k/r$ or $\lim_{r\to\infty}p_r/\log r$ is a positive real number, and similarly for $(q_r)$, and let $(c_r)$ be their union.
\begin{remunerate}
\item If for some integers $k_p$ and $k_q$ we have $s_p=\lim_{r\to\infty}p_r^{k_p}/r \in \reals ^+$ and $s_q=\lim_{r\to\infty}q_r^{k_q}/r \in \reals ^+$, then if $k_p=k_q=k$,
\begin{equation*}
\lim_{r\to\infty}c_r^k/r = (s_1^{-1}+s_2^{-1})^{-1}
\end{equation*}
and (without loss of generality) if $k_p>k_q$,
\begin{equation*}
\lim_{r\to\infty}c_r^{k_p}/r = s_p
\end{equation*}
\item Let $s_1=\lim_{r\to\infty}p_r/\log r$ and $s_2=\lim_{r\to\infty}q_r/\log r$, then if $s_1>0$, $s_2>0$, and at least one of $s_1$ or $s_2$ is finite,
\begin{equation*}
\lim_{r\to\infty}c_r/\log r = \min\{s_1,s_2\}
\end{equation*}
\end{remunerate}
\end{lemma}

{\em Proof.}
Consider the first case.  For $\delta$ with $0<\delta<\min\{s_p^{-1},s_q^{-1}\}$, take $M$ so that for $r\geq M$, $|r/p_r^{k_1}-s_p^{-1}|<\delta$ and $|r/q_r^{k_2}-s_q^{-1}|<\delta$.  For some $c_r$ take $r_1$ and $r_2$ so that $p_{r_1}\leq c_r < p_{r_1+1}$ and $q_{r_2}\leq c_r < q_{r_2+1}$.  Then
\begin{equation*}
r = r_1+r_2
\end{equation*}
and if $c_r$ is large enough then $r_1\geq M$ and $r_2\geq M$.  Then,
\begin{equation*}
c_r^{k_1}(s_p^{-1}-\delta)-1 \leq p_{r_1+1}^{k_1}(s_p^{-1}-\delta)-1 < r_1 < p_{r_1}^{k_1}(s_p^{-1}+\delta) \leq c_r^{k_1}(s_p^{-1}+\delta)
\end{equation*}
and similarly for $r_2$, with $k_2$ and $s_q$ rather than $k_1$ and $s_p$.  Without loss of generality, if $k_1>k_2$ then
\begin{equation*}
\lim_{r\to\infty}c_r^{k_1}/r = s_p
\end{equation*}
and if $k_1=k_2=k$,
\begin{equation*}
\lim_{r\to\infty}c_r^k/r = (s_p^{-1}+s_q^{-1})^{-1}
\end{equation*}
Consider now the second case.  Let
\begin{eqnarray*}
s_p&=&\lim_{r\to\infty} p_r/\log r\\
s_q&=&\lim_{r\to\infty} q_r/\log r
\end{eqnarray*}
and suppose that $s_p$ is finite and that $s_p<s_q$; $s_q$ may be finite or infinite.  For $0<\delta<\min\{s_p^{-1},s_q^{-1}\}$, take $M$ so that for $r\geq M$, $|\log r/p_r - s_p^{-1}|<\delta$ and $|\log r/q_r-s_q^{-1}|<\delta$.  For some $c_r$ let $r_1$ and $r_2$ be defined as in the previous case.  Then
\begin{equation*}
r = r_1+r_2
\end{equation*}
and for $c_r$ large enough, $r_1\geq M$ and $r_2\geq M$.  In this case,
\begin{eqnarray*}
\exp(c_r(s_p^{-1}-\delta))+\exp(c_r(s_q^{-1}-\delta))-2<r< \exp(c_r(s_p^{-1}+\delta)) + \exp(c_r(s_q^{-1}+\delta))\\
\exp(c_r(s_p^{-1}-\delta))(1+\exp(c_r(s_q^{-1}-s_p^{-1})))-2<r< \exp(c_r(s_p^{-1}+\delta))(1 + \exp(c_r(s_q^{-1}-s_p^{-1})))
\end{eqnarray*}
Since $s_p<s_q$ it follows that $s_q^{-1}-s_p^{-1}<0$.  Taking logs, dividing by $c_r$, taking the limit and inverting gives
\begin{equation*}
\lim_{r\to\infty}c_r/\log r = s_p \cvd
\end{equation*}

The result of Theorem \ref{main} now follows from the discussion in Section \ref{secdecomp}, from Lemma \ref{cycle}, from Theorem \ref{irredGraph}, and from repeated application of the rules for composition and union of sequences given by Lemmas \ref{comp} and \ref{union}.

\section{Discussion}
In this paper we have studied the asymptotic scaling of path weights in directed edge-weighted graphs.  Given a starting vertex and an ending vertex, and letting $p_r$ be the weight of the path with $r^{th}$-smallest total weight between starting and ending vertices, we showed that three outcomes are possible: (1) there are finitely many possible paths from start to end, (2) there are infinitely many possible paths and $p_r^c/r \rightarrow s$ for some $c$ and $s$, or (3) there are infinitely many possible paths and $p_r/\log r \rightarrow s$ for some $s$.  Case 1 occurs if and only if the are no strongly connected components reachable from the start vertex and from which the end vertex can be reached.  Case 2 occurs if and only if there is one or more such connected components, but they are all cycles.  Case 3 occurs if and only if at least one of those connected components is not a cycle.  Thus, we can discern the order of the $r$ vs. $p_r$ relationship based on the structure and type of the graph's connected components, and is readily done by standard graph theoretic algorithms.  In cases 2 and 3, determining the constant $s$, and $c$ if relevant, requires analyzing the edge weights.  Again, however, this can be done by well known means, as described above.  We thus have a complete characterization of the asymptotic scaling of path weights for any finite directed graph with positive edge weights.\\

It should be noted that the characterization given in the main result, Theorem \ref{main}, readily applies to Markov chains for which the transition probabilities are positive and $<1$, just by taking the negative $\log$ of the probability.  This gives a graph for which the edge weights are positive.  Moreover, summation of edge weights is equivalent to multiplication of their probabilities, i.e., the sum of positive weights of edges is the negative $\log$ of the product of the transition probabilities along those edges.  Therefore, the results of this paper apply to Markov chains, just by making this transformation.\\

Our initial motivation for studying this problem arose from simulation studies of ``complex" behaviour in randomly-generated continuous-time switching networks.  Dividing the state space of such models into orthants based on the sign of each state variable, we observed that the empirical probabilities of different qualitative return paths to a given orthant were roughly powerlaw distributed (see \cite{Glass05} for some of this work, though the powerlaw relationship in particular was not included in that paper).  We found that a Markov chain model of the transitions between orthants reproduced a similar powerlaw distribution of return paths.  At the time we knew of no theoretical basis for why this should be.  It turns out that that Mandelbrot provided a partial explanation over 50 years ago, while working in the area of coding theory \cite{mandelbrot55}.  Our current result confirms and generalizes Mandelbrot's results.  In the case that the edge weights are the negative logarithms of the transition probability of a Markov chain, then the path weight $p_r$ is the negative log probability of the path, or $-\log Pr(x_r)$, where $x_r$ is the $r^{th}$ most probable path.  If the chain is of the third type described above, then $p_r/\log r  = -\log Pr(x_r)/\log r \rightarrow s$, or $\log Pr(x_r) \approx -s \log r = \log r^{-s}$, so that $Pr(x_r) \approx r^{-s}$.  That is, we have a powerlaw or Zipfian relationship between the path probabilities and the path ranks.  Our work improves on Mandelbrot's result in several ways.  First, it identifies precisely which Markov chains do produce a powerlaw relationship (the case 3 chains) and which do not.  Second, it provides a characterization of the scaling behaviour for the chains that do not generate a powerlaw relationship.  Third, it gives us a means to calculate the exact rate of the scaling ($s$, and possibly $c$), in contrast to Mandelbrot's results, which only established that the relationship exists.\\

There are several important avenues for future research.  Having established the asymptotic scaling of the sequence of weights, it is natural to wonder how quickly the sequence approaches its asymptotic behaviour.  Particularly if we are concerned with some graph derived from a real-world application, it may be important to know whether the asymptotic scaling behaviour is relevant to describing the paths one would see in practice.  To answer this question, it should suffice to examine the subdominant ($2^{nd}$ largest) eigenvalues on strongly connected components, and to relate these to the rate of approach on the whole graph.  Of related concern is that the type of scaling (case 1 vs.\ case 2 vs.\ case 3) can depend on the presence or absence of a single link, because that link may affect the existence or cyclicity of a strongly connected component in the graph.  If we imagine that our weighted graph is derived from a Markov transition matrix, then this means there can be a qualitiative difference between a particular transition probability being zero (hence having no corresponding link in the graph) and that transition probability being $10^{-1000}$.  Yet, in practical terms, a particular event with probability $10^{-1000}$ is likely to never happen in this universe, hence we might as well consider the probability to be zero.  In short, it would be useful to have a characterization of the range of ranks for which the path weights are close to their asymptotic behaviour.\\

Another topic of interest is to relate different sets of path labels.  For example, in the introduction we have already mentioned how roads in road networks might naturally be associated either with their length or with the amount of time it takes to travel.  From the theory we have established, we know that the type of scaling depends only on the graph structure, and not the exact weight values.  Thus, both path lengths and path times must follow the same order of scaling.  But what happens if we look at the lengths of path as ordered by increasing time, or vice versa?  As another example, suppose one set of edge weights corresponds to negative log probabilities of a Markov chain and another set corresponds to something else---a distance, time, cost, etc.  Then establishing a relationship between the two is essentially addressing the probability distribution of path distances, times or costs generated by the chain.  More specifically, the asymptotic relationship would concern the shape of the ``tail" of that distribution.\\

A final topic of interest would be to extend the current results to countable-state graphs.  Some real-world graphs are either very large (e.g., the world-wide web), or come without definite \textit{a priori} bounds on their size (e.g., stock prices), or may even be growing over time---even as paths are being generated on them.  Alternatively, some compact mathematical formalisms (e.g., stochastic grammars describing natural language \cite{manning1999foundations} or stochastic chemical kinetic models \cite{wilkinson2006sms}) implicitly define stochastic processes over countable state spaces.  For examples such as these, it is desirable to establish conditions under which the present results, or some modification of them, may hold.

\bibliography{gr_def}

\begin{thebibliography}{1}

\bibitem{nnmtcs}
A.~Berman and R.J. Plemmons.
\newblock {\em Nonnegative Matrices in the Mathematical Sciences}.
\newblock SIAM, 1994.

\bibitem{Glass05}
L.~Glass, T.~J. Perkins, J.~Mason, H.~T. Siegelmann, and R.~Edwards.
\newblock Chaotic dynamics in an electronic model of a genetic network.
\newblock {\em Journal of Statistical Physics}, 121:969--994, 2005.

\bibitem{horn}
R.A. Horn and C.R. Johnson.
\newblock {\em Matrix Analysis}.
\newblock Cambridge University Press, 1990.

\bibitem{mandelbrot55}
B.B. Mandelbrot.
\newblock On recurrent noise limited coding.
\newblock {\em Information Networks, the Brooklyn Polytechnic Institute
  Symposium}, pages 205--221, 1955.

\bibitem{manning1999foundations}
C.D. Manning and H.Schutze.
\newblock {\em Foundations of Statistical Natural Language Processing}.
\newblock The MIT Press, 1999.

\bibitem{wilkinson2006sms}
D.J. Wilkinson.
\newblock {\em Stochastic modelling for systems biology}.
\newblock Chapman \& Hall/CRC, 2006.

\end{thebibliography}
\bibliographystyle{plain}
\end{document}